\documentclass[a4paper, 12pt]{amsart}

\usepackage[utf8]{inputenc}
\usepackage[T1]{fontenc}
\usepackage{lmodern}
\usepackage{mathrsfs}
\usepackage[english]{babel}
\usepackage[size=2em,nohug,midshaft,scriptlabels,PostScript=dvips]{diagrams}
\usepackage{url}

\theoremstyle{plain}
\newtheorem{theorem}{Theorem}[section]
\newtheorem{lemma}[theorem]{Lemma}
\newtheorem{proposition}[theorem]{Proposition}
\newtheorem{corollary}[theorem]{Corollary}
\theoremstyle{definition}
\newtheorem{definition}[theorem]{Definition}
\newtheorem{remark}[theorem]{Remark}

\newcommand{\ZZ}{\mathbb{Z}} 
\newcommand{\iso}{\cong}      
\newcommand{\PP}{\mathbb{P}}  
\newcommand{\VV}{\mathbb{V}}  
\renewcommand{\AA}{\mathbb{A}}  
\newcommand{\sheaf}[1]{\mathscr{#1}} 
\newcommand{\OO}{\sheaf{O}}   
\newcommand{\res}[2]{\left.#1\right|_{#2}} 
\newcommand{\Gm}{\mathbb{G}_m} 
\newcommand{\tensor}{\otimes} 
\newcommand{\stab}{\mathit{s}} 
\newcommand{\sstab}{\mathit{ss}} 

\DeclareMathOperator{\Hom}{Hom}
\DeclareMathOperator{\Spec}{Spec}
\DeclareMathOperator{\Proj}{Proj}
\DeclareMathOperator{\Hilb}{Hilb}
\DeclareMathOperator{\Sym}{Sym}

\newcommand{\Kps}{K[\![t]\!]} 
\newcommand{\Kls}{K(\!(t)\!)} 

\title{A relative Hilbert--Mumford criterion}

\begin{document}

\author{Martin G. Gulbrandsen}
\address{University of Stavanger\\
Department of Mathematics and Natural Sciences\\
4036 Stavanger\\
Norway}
\email{martin.gulbrandsen@uis.no}

\author{Lars H. Halle}
\address{University of Copenhagen\\ 
Department of Mathematical Sciences\\
Universitetsparken 5\\ 
2100 Copenhagen\\ 
Denmark}
\email{larshhal@math.ku.dk}

\author{Klaus Hulek}
\address{Leibniz Universit\"at Hannover\\
Institut f\"ur algebraische Geometrie\\
Welfengarten 1\\
30060 Hannover\\
Germany}
\email{hulek@math.uni-hannover.de, hulek@ias.edu}
\curraddr{Institute for Advanced Study\\
School of Mathematics\\
1 Einstein Drive\\
Princeton, NJ 08540\\
USA}


\begin{abstract}
We generalize the classical Hilbert--Mumford criteria for GIT (semi-)stability
in terms of one parameter subgroups of a linearly reductive group  $G$ over a
field $k$, to the relative situation of an equivariant, projective morphism
$X\to \Spec A$ to a noetherian $k$-algebra $A$.  We also extend the classical
projectivity result for GIT quotients: the induced morphism $X^\sstab/G \to
\Spec A^G$ is projective.  As an example of applications to moduli problems, we
consider degenerations of Hilbert schemes of points.
\end{abstract}

\maketitle

\section{Introduction}

Geometric invariant theory (GIT) is one of the most important tools in
algebraic geometry. It was Mumford's seminal book \cite{GIT} which brought the
classical field of ``invariants'' into modern algebraic geometry. In particular
moduli spaces are often constructed as quotients of a (quasi-projective)
scheme $X$ by a (reductive) algebraic group $G$. However, for the
quotient to have good properties it is essential to restrict  to (semi-)stable
points in $X$. In applications, the most important result of Mumford's
theory  is perhaps the numerical characterization of (semi-)stable points
in terms of $1$-parameter subgroups ($1$-PS), given in \cite[Theorem 2.1]{GIT}.
The context of this theorem is that $X$ is a projective variety, $\sheaf{L}$ is
an ample invertible sheaf on $X$ and the action of $G$ on $X$ is $G$-linearized on
$\sheaf{L}$.

Motivated by studying degenerations of Hilbert schemes we were led to consider
the following relative situation: Let $S$ be an
affine, noetherian scheme over an arbitrary field $k$, and let $X$ be a relatively projective
scheme over $S$. Let $G$ be a linearly reductive group over $k$, acting on $X$
and on $S$, such that the structural morphism $f\colon X\to S$ is equivariant.
Finally, let $\sheaf{L}$ be an ample, $G$-linearized invertible sheaf on $X$.
Of course, as $S$ is affine, ampleness is equivalent to relative ampleness over
$S$.

Our aim is to find a numerical characterization, analogous to Mumford's theorem
\cite[Theorem 2.1]{GIT}, for the (semi-)stable points in $X$, with respect to
the $G$-linearized sheaf $\sheaf{L}$. Consider the induced action of a one
parameter subgroup $\lambda\colon \Gm\to G$, applied to a point $p$ in $X$. If
the limit $p_0$ exists in $X$, as $t\in \Gm$ goes to zero, then the invariant
$\mu^{\sheaf{L}}(\lambda, p)$ is defined as the negative of the weight of the
$\Gm$-action on the fibre $\sheaf{L}(p_0)$.  In the absolute situation, where
$X$ is projective over $k$, the limit $p_0$ always exists.  In our relative
situation, we find that if $p_0$ does not exist, then the one parameter
subgroup in question can be ignored; formally we let $\mu^{\sheaf{L}}(\lambda,
p) = \infty$ in this case. Our criterion is then as follows (see Section
\ref{sec:statement}):

\newtheorem*{theo:numericalcriterion}{Theorem \ref{theo:numericalcriterion}}
\begin{theo:numericalcriterion}
Let $k$ be an arbitrary field and let $f\colon X\to S$ be a projective morphism
of $k$-schemes.  Assume $S=\Spec A$ is noetherian and affine.  Let $G$ be an
affine, linearly reductive group over $k$, acting on $X$ and on $S$ such that
$f$ is equivariant, and let $\sheaf{L}$ be an ample, $G$-linearized invertible
sheaf on $X$.  Let $p\colon \Spec \Omega \to X$ be a geometric point.
\begin{enumerate}
\item The geometric point $p$ is stable if and only if
$\mu^{\sheaf{L}}(\lambda, p)>0$ for every non trivial $1$-PS
$\lambda\colon \mathbb{G}_{m, K} \to G_K$ and every field extension
$\Omega \subset K$.
\item The geometric point $p$ is semi-stable if and only if
$\mu^{\sheaf{L}}(\lambda, p)\ge 0$ for every $1$-PS
$\lambda\colon \mathbb{G}_{m, K} \to G_K$ and every field extension
$\Omega \subset K$.
\end{enumerate}
Moreover, if $S$ is of finite type over $k$, then the above statements continue
to hold with just $K=\Omega$, i.e.~it suffices to test with $1$-PSs defined
over $\Omega$.
\end{theo:numericalcriterion}

We state also the special case of a finite type situation:

\begin{corollary}\label{cor:finitetype}
Assume the base scheme $S$ is of finite type over an algebraically closed field
$k$.  Let $p\in X$ be a closed point.
\begin{enumerate}
\item The point $p$ is stable if and only if $\mu^{\sheaf{L}}(\lambda, p) > 0$
for every nontrivial $1$-PS $\lambda\colon \Gm \to G$.
\item The point $p$ is semi-stable if and only if   $\mu^{\sheaf{L}}(\lambda, p) \geq 0$
for every $1$-PS $\lambda\colon \Gm \to G$.
\end{enumerate}
\end{corollary}

Already the special case $X=S=\AA^n$ is interesting: we recover King's
numerical criterion for (semi-)stability in affine space \cite{Kin}. In this
situation, $\sheaf{L}$ is necessarily trivial as an invertible sheaf, but its
$G$-linearization may be nontrivial.

Our approach to proving the main theorem is to go back to Mumford's argument.
This, however, is not straightforward.  In the projective case any $\Gm$-action
can be diagonalized, which makes it easy to compute  the invariant
$\mu^{\sheaf{L}}(\lambda, p)$.  In our situation, we first embed $X$
equivariantly into a projective fibre space $\PP(V)\to S$, where $V$ is a
coherent sheaf on $S$ with a $G$-action.  The main challenge is then that we do
not know a suitable diagonalization result for $\Gm$-actions on $\PP(V)$. Our
solution is to work directly with the isotypical decomposition $V=\bigoplus_d
V_d$, which serves as a replacement for a diagonalizing basis.  Going through
the various steps carefully we see that the base $S$ need not necessarily be of
finite type, but that it suffices to assume $S$ to be noetherian. This is
indeed a major advantage for moduli problems.  Here $X \to S$ would typically
be an $S$-valued point of some moduli functor, and the condition that $S$ be of
finite type would be unsatisfactory.

It was pointed out to us by A.~Schmitt that in the (more special)  situation of
Corollary \ref{cor:finitetype} there is another approach.  Indeed, it is
natural to reduce to $\PP^m\times\AA^n$ by finding first an equivariant
embedding of $S$ into $\AA^n$, and then, using $\sheaf{L}$, an equivariant
embedding of $X$ into $\PP^m\times\AA^n$.  The required criterion for
$\PP^m\times\AA^n$ can then be proved by compactifying to $\PP^m\times\PP^n$
and applying Schmitt's criterion \cite[Prop.~2.9]{Sch} for a product of
projective spaces.

Equivariant GIT has been considered in the literature by various authors in
different settings: Reichstein \cite{Rei} treats equivariant morphisms $f\colon
X\to Y$ between projective varieties, and Hu \cite{Hu} uses symplectic
techniques to study equivariant morphisms between quasi-projective varieties.
We also mention that Seshadri \cite{Sesh2} develops GIT over an arbitrary
noetherian ring, that is, in the situation where $G$ acts trivially on the base
$S=\Spec A$.

We shall now briefly outline the structure of the paper. In Section
\ref{sec:GIT} we recall the notions of stable and semi-stable points and the
existence of good and geometric quotients, closely following \cite{GIT}. The
main point of this section is to establish that, although Mumford works in the
context of algebraic (pre)schemes, only the noetherian property is used when
constructing the quotient.  We also show that if $X$ is projective over $S$,
then the quotient is projective over $S/G = \Spec A^G$ (Proposition
\ref{prop:proj}).  In Section \ref{sec:statement} we formally introduce the
invariants $\mu^{\sheaf{L}}(\lambda, p)$ and state the main result (Theorem
\ref{theo:numericalcriterion}).  In Section \ref{sec:criteria} we rephrase the
property of (semi-)stability using a  $G$-equivariant embedding of $X$ into a
projective fibre space  $\PP(V)$. We will then use this reformulation in the
proof of the main result.  Finally in Section \ref{sec:numerical} we prove our
numerical characterization for (semi-)stable points. It is here that we
reinterpret the invariants $\mu^{\sheaf{L}}(\lambda, p)$ in terms of the
grading $V=\bigoplus_{d\in \ZZ} V_d$. We conclude the paper by discussing two
examples in Section \ref{sec:examples}.  In the first basic example we consider
a simple action of $\Gm$ on $\PP^1 \times \AA^2 \to \AA^2$.  In the second
example we consider  a family $X \subset \PP^2 \times \AA^1 \to  \AA^1$ which
is the degeneration of a smooth conic into a pair of (different) lines. Here we
discuss a good compactification of the degree $2$ Hilbert scheme
$\Hilb^2_{X^{\mathrm{sm}}/\AA^1}$, where $X^{\mathrm{sm}}$ denotes the smooth
locus of the morphism $f\colon X \to \AA^1$ given by projection to the second
factor. We have good control over an invertible sheaf on the relative Hilbert
scheme giving Grothendieck's projective embedding, whereas the embedding itself
is unwieldy. Hence we prefer the language of linearized invertible sheaves over
equivariant projective embeddings.  This is a  toy example of the situation
which we investigate in a forthcoming paper where we consider degenerations of
Hilbert schemes of points on $K3$-surfaces.

\subsection*{Acknowledgments}
We thank Friedrich Knop, Peter Littelmann and David Rydh (to whom we owe Remark
\ref{rem:inv-scheme}) for useful discussions. The first author would like to
thank NFR for partial support under grant 230986. The third author would like
to thank DFG for partial support under grant Hu 337/6-2  and also the Fund for
Mathematics to the  Institute of Advanced Study in  Princeton, which provided
excellent working conditions.

\subsection{Notation and preliminaries}

We list here some notation and facts used throughout the paper.

If $\sheaf{F}$ is a coherent sheaf on $S$, we use Grothendieck's contravariant notations:
\begin{align*}
\PP(\sheaf{F}) &= \Proj \Sym \sheaf{F}\\
\VV(\sheaf{F}) &= \Spec \Sym \sheaf{F}
\end{align*}
The sheaf of sections of $\VV(\sheaf{F})$ is $\sheaf{F}^\vee$.

Let $Y$ be a $k$-scheme and assume that $G$ acts on $Y$.  For a $k$-scheme $T$
and $T$-valued points $p\colon T\to Y$ and $g\colon T\to G$, we write $g\cdot
p\colon T\to Y$ for the result of acting by $g$ on $p$.  We denote by $G \cdot
p$ the orbit of $p$. It is a set in $Y \times_k T$. We denote by $G_p$ the
stabilizer of $p$. This is a subgroup scheme of $G \times_k T$.  The action is
\emph{closed} if the orbits of all geometric points are closed.

Let $k\subset K$ be a field extension. Then $G$ is linearly reductive if and
only if $G_K$ is linearly reductive.

\section{The GIT quotient}  \label{sec:GIT}

We recall the notions of stable and semi-stable points and the existence
theorem for good and geometric quotients, following Mumford \cite{GIT}.  The
main point is to establish that although Mumford works in the context of
algebraic (pre)schemes, only the noetherian property is used when constructing
the quotient. Moreover, in Proposition \ref{prop:proj}, we extend the classical
projectivity result for GIT quotients to the relative case of an equivariant
projective morphism $X\to S$ to a noetherian affine base scheme $S$.

All schemes and morphisms in this section are defined over the field $k$.  Let
$G$ be an affine, linearly reductive algebraic group and let $X$ be a scheme
equipped with an action $\sigma\colon G\times X\to X$.  A morphism $\phi\colon
X\to Y$ to another scheme $Y$ is \emph{invariant} if $\phi\circ\sigma =
\phi\circ p_2$ as morphisms $G\times X\to Y$.

The notions of good and geometric quotients are well established, but as
variations do occur, we give the definitions we shall use.

\begin{definition}[{Mumford \cite[Remark 6, §0.2]{GIT}, Seshadri \cite[Def.~1.5]{Sesh}}]
A morphism $\phi\colon X\to Y$ is a \emph{good quotient} if it is invariant and
affine, and
\begin{itemize}
\item[(i)] for each closed invariant subscheme $W\subset X$, the image
$\phi(W)$ is a closed subset of $Y$,
\item[(ii)] for each collection $\{W_i\}_{i\in I}$ of closed invariant
subschemes $W_i\subset X$, we have an equality
\begin{equation*}
\phi({\textstyle\bigcap}_i W_i) = {\textstyle\bigcap}_i \phi(W_i)
\end{equation*}
of schemes,
\item[(iii)] the sheaf homomorphism $\phi^*\colon \OO_Y \to \phi_*\OO_X$ is
injective and its image is the invariant subsheaf $(\phi_*\OO_X)^G$.
\end{itemize}
\end{definition}

\begin{remark}\label{rem:inv-scheme}
If $W\subset X$ is an invariant closed subset, then there exists an invariant
subscheme structure on $W$. In fact, first let $W_{\mathit{red}}$ be the
reduced closed subscheme.  Then give $W$ the scheme structure as the schematic
image of the morphism $G\times W_{\mathit{red}}\to X$.
\end{remark}

\begin{definition}[{Mumford \cite[Def.~0.6]{GIT}}]
A morphism $\phi\colon X\to Y$ is a \emph{geometric quotient} if it is good and
the morphism $(\sigma, p_2)\colon G\times X \to X\times_Y X$ is surjective.
\end{definition}

A morphism $\phi\colon X\to Y$ is \emph{universally} geometric or good if it
stays geometric or good after arbitrary base change.

Let $\sheaf{L}$ be an invertible sheaf on $X$ equipped with a
$G$-linearization.

\begin{definition}[{Mumford \cite[Def.~1.7]{GIT}}]
Let $p$ be a geometric point in $X$.
\begin{enumerate}
\item $p$ is \emph{semi-stable} with respect to $ \sheaf{L} $ if there is a
section $\sigma\in H(X,\sheaf{L}^n)^G$ for some $n>0$ such that $X_{\sigma}$ is
affine and $\sigma(p)\ne 0$.
\item $p$ is \emph{stable} with respect to $\sheaf{L}$ if there is a section
$\sigma$ as above, such that, in addition, the action of $G$ on $X_{\sigma}$ is
closed and the stabilizer group of every point in $X_{\sigma}$ is finite.
\end{enumerate}
We denote by $X^\sstab(\sheaf{L})$ and $X^\stab(\sheaf{L})$ the open subsets of
$X$ whose geometric points are the semi-stable, respectively stable, geometric
points in $X$.
\end{definition}

A stable point in the above sense is called \emph{properly} stable in the
terminology of \cite{GIT}.

\begin{theorem}\label{thm:git-quotient}
Let $G$ be an affine, linearly reductive algebraic group over a field $k$.  Let
$X$ be a noetherian scheme over $k$, equipped with a $G$-action and a
$G$-linearized invertible sheaf $\sheaf{L}$.
\begin{enumerate}
\item There exists a universally good quotient $\phi\colon X^\sstab(\sheaf{L})
\to Y$.  Moreover, there is an ample invertible sheaf $\sheaf{M}$ on $Y$ such
that $\phi^*\sheaf{M}\iso \sheaf{L}^{\tensor n}$ for some $n>0$.
\item There is an open subset $\widetilde{Y}\subseteq Y$ such that
$X^\stab(\sheaf{L}) = \phi^{-1}(\widetilde{Y})$ and such that the restriction
$X^\stab(\sheaf{L}) \to \widetilde{Y}$ of $\phi$ is a universally geometric
quotient.
\end{enumerate}
\end{theorem}

\begin{proof}
This is essentially Mumford's Theorem 1.10 \cite{GIT}. In the statement of that
theorem, $X$ is assumed to be of finite type, but the proof uses only that $X$
is noetherian. For use in Proposition \ref{prop:proj} below, we recall briefly
the main points in the construction of $Y$ and $\sheaf{M}$.

The semi-stable locus $X^\sstab$ is the union of all affine open subsets of the
form $X_\sigma$, where $\sigma$ runs through $\Gamma(X,\sheaf{L}^n)^G$ for all
$n>0$.  Since $X$ is noetherian, $X^\sstab$ can be covered by finitely many
such subsets $U_i = X_{\sigma_i}$, and we may suppose that all $\sigma_i\in
\Gamma(X,\sheaf{L}^n)^G$ for a common $n>0$.  Suppose $U_i = \Spec R_i$ and let
$V_i = \Spec R_i^G$.  By \cite[Theorem 1.1]{GIT} and its proof, the canonical
morphism $\phi_i\colon U_i \to V_i$ is a good quotient, and for every fibre
diagram
\begin{equation*}
\begin{diagram}
U_i' & \rTo & U_i \\
\dTo^{\phi_i'} && \dTo_{\phi_i}\\
V_i' & \rTo & V_i
\end{diagram}
\end{equation*}
where $V_i'$ is affine, also $\phi_i'\colon U_i \to V_i'$ is an affine
quotient, i.e.~induced by an inclusion of the form $A^G \subset A$, so it is
again good.  It follows that the $\phi_i$'s can be glued to form $\phi\colon
X^\sstab \to Y$, and (using that being good is Zariski local on the target),
that $\phi$ is universally good.

The ample invertible sheaf $\sheaf{M}$ is given by a cocycle $\{\sigma_{ij}\}$
for the open cover $\{V_i\}$ of $Y$, constructed as follows: the quotient
$\sigma_j/\sigma_i$ is a $G$-invariant regular function on $U_i$, hence is a
regular function on $V_i$.  Its restriction to $V_i\cap V_j$ is $\sigma_{ij}$.
Ampleness is proved using the criterion [EGAII, Théorème 4.5.2(a')], which
applies when $Y$ is noetherian, and this holds since $Y$ is constructed by
glueing together a finite number of spectra of noetherian rings.
\end{proof}

We shall now state a proposition which asserts that if $X$ is relatively
projective, then the same holds for the quotient.  In the classical situation
$S=\Spec k$, see \cite[the paragraph before Converse 1.12]{GIT} or
\cite[Theorem 1.1]{Sesh}.

\begin{proposition}\label{prop:proj}
Let $G$, $X$, $\sheaf{L}$, $Y$ and $\sheaf{M}$ be as in Theorem
\ref{thm:git-quotient}.  Assume moreover that $\sheaf{L}$ is ample, $X$ is
relatively projective over an affine noetherian scheme $S = \Spec A$, and $G$
acts on $S$ in such a way that the projection $X\to S$ is equivariant.
\begin{enumerate}
\item Let $T$ be the ring $\bigoplus_{d\ge 0} H^0(X, \sheaf{L}^d)$ of sections.
Then there are isomorphisms $Y\iso \Proj T^G$ and $\sheaf{M} \iso \OO_Y(n)$ for
some $n$.
\item $Y$ is relatively projective over $\Spec A^G$.
\end{enumerate}
\end{proposition}

Note that, as $S$ is affine, an invertible sheaf $\sheaf{L}$ on $X$ is ample if
and only if it is relatively ample over $S$.

\begin{proof}
We note that the section ring $T^{(n)} = \bigoplus_{d\ge 0}
H^0(X,\sheaf{L}^{dn})$ for an appropriate tensor power $\sheaf{L}^n$ is
finitely generated as a $T_0$-algebra and $X \iso \Proj T$ over $T_0$
\cite[Lemma 8.1.23]{Liu}.  We also note that $T_0$ is finite as an $A$-module,
as $X$ is proper over $A$.

Using essentially Hilbert's argument for finite generation of invariant rings,
we will prove:
\begin{itemize}
\item[(i)] $T_0^G$ is a finite $A^G$-module
\item[(ii)] There exists an integer $n>0$ such that the Veronese subring
$(T^G)^{(n)}$ of the invariant ring $T^G$ is finitely generated as a
$T_0^G$-algebra.
\end{itemize}
Claim (ii) implies that $\Proj T^G$ is projective over $\Spec T_0^G$.  Claim
(i) then implies that $\Proj T^G$ is projective over $ \Spec A^G$ as well.

To prove (i), consider the $A$-submodule $AT_0^G\subset T_0$ generated by
$T_0^G$.  As $T_0$ is finite as an $A$-module and $A$ is noetherian, $T_0$ is a
noetherian $A$-module, so $AT_0^G$ is finitely generated, and we may choose a
finite generating set $t_1,\dots,t_r$ in $T_0^G$. We claim that these elements
in fact generate $T_0^G$ as an $A^G$-module. For let $t\in T_0^G$, then
\begin{equation*}
t = a_1t_1 + \cdots + a_rt_r
\end{equation*}
with $a_i\in A$.  Apply the Reynolds operator $E$ to obtain
\begin{equation*}
t = E(t) = E(a_1)t_1 + \cdots + E(a_r)t_r
\end{equation*}
by the Reynolds identity. Since $E(a_i)\in A^G$, we are done.

For (ii), we first replace $T$ by $T^{(n)}$ if necessary, to ensure that $T$ is
finitely generated as a $T_0$-algebra. Since $A$ is a noetherian ring, so is
$T_0$ and thus also $T$.  Let $J\subset T$ be the ideal generated by all
homogeneous, $G$-invariant elements of positive degree.  As $T$ is noetherian,
$J$ is finitely generated, say by homogeneous elements $f_1,\dots, f_m\in T^G$.
Consider the subalgebra $T_0^G[f_1,\dots,f_m]\subseteq T^G$.  Let $f\in T^G$ be
a homogeneous element of degree $d$. We claim that $f\in T^G_0[f_1,\dots,f_m]$.
This is trivial for $d=0$.  If $d>0$ then $f\in J$, so we may write
\begin{equation*}
f = h_1f_1 + \cdots + h_m f_m,\quad h_i\in T
\end{equation*}
and we may assume all $h_i$ are homogeneous of degree $<d$. Apply the Reynolds
operator $E$ to obtain
\begin{equation*}
f = E(h_1) f_1 + \cdots + E(h_m) f_m
\end{equation*}
by the Reynolds identity.  Now $E$ preserves degree, so $E(h_i)$ are invariant
of degree $<d$.  By induction on $d$, we are done.

It remains to exhibit the isomorphisms $Y\iso \Proj T^G$ and $\sheaf{M}\iso
\OO_Y(n)$.  Since $X=\Proj T$, the affine open subset $U_i$ in the proof of
Theorem \ref{thm:git-quotient} is $\Spec T_{(\sigma_i)}$ and thus the quotient
$V_i$ is $\Spec T_{(\sigma_i)}^G$. On the other hand, $\Proj T^G$ is covered by
affine open subschemes $\Spec T_{(\sigma)}^G$, where $\sigma$ runs through
$T^G_n = \Gamma(X,\sheaf{L}^n)^G$ for all $n>0$. As $T^G$ is finitely
generated, $\Proj T^G$ can be covered by finitely many such open affine
subschemes $V_i = \Spec T^G_{(\sigma_i)}$, and we may then assume $\sigma_i \in
T^G_n$ for a common $n>0$. These open affine subschemes are glued along
$U_i\cap U_j = \Spec T^G_{(\sigma_i\sigma_j)}$ via the identifications
$T^G_{(\sigma_i\sigma_j)} = (T^G_{(\sigma_i)})_{\sigma_j/\sigma_i}$. But this
is precisely the construction of $Y$ in Theorem \ref{thm:git-quotient}.
Moreover, the invertible sheaf $\OO(n)$ on $\Proj T^G$, or equivalently
$\OO(1)$ on $\Proj (T^G)^{(n)}$, corresponds to the cocycle $\{\sigma_{ij} =
\sigma_j/\sigma_i\}$ for the open cover $\{V_i\}$.  Thus $\OO(n)$ agrees with
$\sheaf{M}$.
\end{proof}

We quote the following two results from \cite{GIT}:

\begin{proposition}\label{prop:variations}
The following are equivalent conditions on a geometric semi-stable point
$p\colon \Spec \Omega \to X$:
\begin{enumerate}
\item $p$ is stable.
\item The stabilizer group $G_p$ is finite and the orbit $G\cdot p$ is closed
in $X^\sstab(\sheaf{L})_\Omega$.
\item The stabilizer group $G_p$ is finite and there is a section $\sigma \in
H^0(X,\sheaf{L}^n)^G $for a suitable $n > 0$ such that the orbit $G \cdot p$
is closed in $(X_\sigma)_\Omega$.
\end{enumerate}
\end{proposition}

\begin{proof}
This is \cite[Amplification 1.11]{GIT} (with ``regular'' replaced by ``finite
stabilizer'', since we work with properly stable points).  Although Mumford's
$X$ is of finite type, the proof goes through without change as long as Theorem
\ref{thm:git-quotient} on the existence of quotients holds, which we have
established in the noetherian situation.
\end{proof}

\begin{lemma}\label{lemma:stab-extension}
Let $k\subset K$ be an arbitrary field extension, and let $\sheaf{L}_K$ be the
pullback of $\sheaf{L}$ to $X_K$. Then
$X^\stab(\sheaf{L})_K = X_K^\stab(\sheaf{L}_K)$ and
$X^\sstab(\sheaf{L})_K = X^\sstab_K(\sheaf{L}_K)$.
\end{lemma}

\begin{proof}
This is \cite[Proposition 1.14]{GIT}, where $\sheaf{L}$ is not assumed to be
ample.  In the ample case, where the open subsets $X_{\sigma}\subset X$ for
$\sigma\in H^0(X,\sheaf{L}^d)$ are automatically affine, this is easier, so we
give the proof in this case, arguing as in \cite[Amplification 1.11]{GIT}.

Because of the characterization in Proposition \ref{prop:variations}(2) of
stable points, among semi-stable ones, it suffices to show the semi-stable
case.

We have
\begin{equation*}
H^0(X_K,\sheaf{L}_K^d) = H^0(X,\sheaf{L}^d)\tensor_k K.
\end{equation*}
If $p$ is a geometric point in $X^\sstab_K(\sheaf{L}_K)$, then there is an
invariant section $\sigma = \sum_i \sigma_i \tensor \alpha_i$ which is nonzero
on $p$. We may assume the elements $\alpha_i\in K$ are linearly independent;
then it follows that each $\sigma_i$ is invariant.  At least one $\sigma_i$ is
nonzero on $p$, hence $p$ is in $X^\sstab(\sheaf{L})_K$.  The other inclusion
$X^\sstab(\sheaf{L})_K \subset X^\sstab_K(\sheaf{L}_K)$ is trivial.
\end{proof}

\section{The numerical criterion --- statement}\label{sec:statement}

Let $p\colon \Spec \Omega \to X$ be a geometric point and let $\lambda\colon
\mathbb{G}_{m, K} \to G_K$ be a $1$-parameter subgroup over some field
extension $\Omega\subset K$.  Action by $\Gm$ on $p$ defines a morphism
$\mathbb{G}_{m,K} \to X_K$, and we may ask whether this extends to $\AA^1_K \to
X_K$. If it does, it does so uniquely, as $X$ is separated, and restriction to
$0\in \AA^1_K$ defines a $K$-valued limit point $p_0\colon \Spec K \to X_K$,
which we will denote
\begin{equation*}
p_0 = \lim_{t\to 0} \lambda(t)\cdot p.
\end{equation*}
It is a fixed point for the action of $\mathbb{G}_{m,K}$.

\begin{definition}\label{def:stabililty}
In the above situation, we let $\mu^{\sheaf{L}}(\lambda, p)$ denote the
negative of the $\mathbb{G}_{m,K}$-weight on the fibre $\sheaf{L}(p_0)$,
provided the limit point $p_0$ exists.  Otherwise, we let
$\mu^{\sheaf{L}}(\lambda, p) = \infty$.
\end{definition}

\begin{remark}
In order to simplify notation we shall sometimes write $\mu(\lambda, p)$
instead of $\mu^{\sheaf{L}}(\lambda, p)$ when it is clear which line bundle
$\sheaf{L}$ we refer to.
\end{remark}

\begin{theorem}\label{theo:numericalcriterion}
Let $k$ be an arbitrary field and let $f\colon X\to S$ be a projective morphism
of $k$-schemes.  Assume $S=\Spec A$ is noetherian and affine.  Let $G$ be an
affine, linearly reductive group over $k$, acting on $X$ and on $S$ such that
$f$ is equivariant, and let $\sheaf{L}$ be an ample, $G$-linearized invertible
sheaf on $X$.  Let $p\colon \Spec \Omega \to X$ be a geometric point.
\begin{enumerate}
\item The geometric point $p$ is stable if and only if
$\mu^{\sheaf{L}}(\lambda, p)>0$ for every non trivial $1$-PS
$\lambda\colon \mathbb{G}_{m, K} \to G_K$ and every field extension
$\Omega \subset K$.
\item The geometric point $p$ is semi-stable if and only if
$\mu^{\sheaf{L}}(\lambda, p)\ge 0$ for every $1$-PS
$\lambda\colon \mathbb{G}_{m, K} \to G_K$ and every field extension
$\Omega \subset K$.
\end{enumerate}
Moreover, if $S$ is of finite type over $k$, then the above statements continue
to hold with just $K=\Omega$, i.e.~it suffices to test with $1$-PSs defined
over $\Omega$.
\end{theorem}

Note that, to detect (semi-)stability, it suffices to test with one parameter
subgroups for which the limit point $p_0$ in Definition \ref{def:stabililty}
exists, since the value of $\mu$ is infinite otherwise.  Since $f\colon X\to S$
is relatively projective, the valuative criterion for properness shows that
the limit point $p_0$ in $X$ exists if and only if the limit $\lim_{t\to 0}
\lambda(t)\cdot f(p)$ exists in $S$.

We may view the geometric point $p$ as a closed point in $X_\Omega$.  Since
(semi-)stability is preserved by the base change to $\Omega$ by Lemma
\ref{lemma:stab-extension}, we may thus, without loss of generality, assume
that $k$ is algebraically closed and $p$ is a closed point. We assume this in
the remainder of this paper.

\section{Geometric criteria} \label{sec:criteria}

We continue to work with a projective, $G$-equivariant morphism $f\colon X\to
S$, where $S=\Spec A$ for a noetherian $k$-algebra $A$, and an ample,
$G$-linearized invertible sheaf $\sheaf{L}$ on $X$.  By replacing $\sheaf{L}$
with a tensor power, which does not affect the stable or semi-stable loci in
$X$, we may assume that $\sheaf{L}$ is $f$-very ample.

Let $V=H^0(X,\sheaf{L})$, considered as an $A$-module, and let
\begin{equation*}
T = \bigoplus_{d\ge 0} H^0(X,\sheaf{L}^d)
\end{equation*}
be the section ring of $\sheaf{L}$. After replacing $\sheaf{L}$ by a further
tensor power if necessary, we may assume that $T$ is finitely generated by
$T_1=V$ as a $T_0$-algebra, and $X=\Proj T$ \cite[Lemma 8.1.23]{Liu}.  We let
\begin{equation*}
X^*=\Spec T
\end{equation*}
be the affine cone over $X$. The surjection $\Sym V \to T$ induces closed
embeddings $X\subset \PP(V)$ and $X^*\subset\VV(V)$.  We write $0_S$ for the
zero section in $\VV(V)$.

Our aim in this section is to characterize (semi-)stability of a point in
$X\subset\PP(V)$ in terms of the orbit of a lift to $\VV(V)$; these are the
geometric criteria for (semi-)stability.

First we need a lemma: let $W$ be a $k$-scheme and assume $G$ acts on $W$.
Recall that for every point $w\in W(k)$, there is an induced map
\begin{equation*}
\phi_w \colon G  \to W
\end{equation*}
defined by $ g \mapsto g \cdot w $.

\begin{lemma}\label{lemma:proper-phi}
Let $W$ be a separated $k$-scheme, let $ \sigma : G \times W \to W $ be an
algebraic action and let $w\in W$ be a closed point.  Then
\begin{equation*}
\phi_w\colon G \to W
\end{equation*}
is proper if and only if $G \cdot w$ is closed in $W$ and the stabilizer $G_w$
is finite.
\end{lemma}

See \cite[Lemma 0.3]{GIT} and \cite[Lemma 3.17]{New}; the latter is phrased for
varieties, but the proof goes through in wide generality, and it is the one we
follow below.

\begin{proof}
If $\phi_w$ is proper, its image $G \cdot w$ is closed in $W$. Moreover, $G_w$
is proper over $k$ and is closed (being the fibre over the closed point $w$) in
$G$, which is affine and of finite type. Thus $G_w$ is finite.

Assume now that $G \cdot w$ is closed in $W$ and that $G_w$ is finite.  Let
$Z\subset W$ be the schematic image of the (quasi-compact) morphism $\phi_w$.
Since $G$ is of finite type over $k$, the set $G(k)$ is dense in $G$, and its
image forms a dense subset of $Z$. As $Z$ is an invariant subscheme, we may
replace $W$ by $Z$.

The map $\phi_w$ is quasi-finite, since this holds locally at any point in
$G_w$, and thus everywhere by translation. By \cite[Lemma \texttt{03I1}]{Stacks}, there
exists an open dense $U \subset Z$ such that the restriction $\phi_w^{-1}(U)
\to U$ is finite. The translates of $U$ must cover $Z$, so $\phi_w$ is indeed
finite and hence proper.
\end{proof}

\begin{proposition}\label{prop:geometric-criteria}
Let $p\in X$ be a closed point and let $p^*\in X^*$ denote any lift.
Denote by $\mathrm{cl}(G \cdot p^*)$ the closure of the orbit of $p^*$ in $\VV(V)$.
\begin{enumerate}
\item $p \in X^\sstab$ if and only if $0_S \cap \mathrm{cl}(G \cdot p^*) = \emptyset$ inside $\VV(V)$.
\item $p \in X^\stab$ if and only if
 $\phi_{p^*}\colon G\to \VV(V)$ is proper.
\end{enumerate}
\end{proposition}

\begin{proof}
This is a direct extension of \cite[Prop.~2.2]{GIT}, and we just check that
the proof given by Mumford goes through in the relative situation.

Let $Z = 0_S \cap X^*$.

(1) If $p\in X^\sstab$, there is an invariant section $\sigma$ of
$\sheaf{L}^d$, with $d>0$, such that $\sigma(p)\ne 0$. As $\sigma$ is
invariant, it is constant (and nonzero, as $\sigma(p)\ne 0$) on the orbit
$G\cdot p^*$, hence on its closure.  But $\sigma$ is zero on $Z$, since $d>0$.
Thus $\mathrm{cl}(G\cdot p^*)$ cannot intersect $0_S$.

For the converse, observe that since $Z$ and $\mathrm{cl}(G \cdot p^*)$ are
closed and invariant in $X^*$, the reductivity of $G$ implies that we can find
an invariant element $\sigma\in T^G$ such that $\sigma \equiv 1$ on
$\mathrm{cl}(G \cdot p^*)$ and $ \sigma \equiv 0 $ on $Z$. This implies that
some homogeneous component $\sigma_d\in T_d = H^0(X,\sheaf{L}^d)$ of $\sigma$
of degree $d>0$ satisfies $\sigma(p^*)\ne 0$.  As the $G$-action respects the
grading on the coordinate ring, this component $\sigma_d$ is invariant.

(2) Assume $\phi_{p^*}$ is proper. In particular the orbit $G\cdot p^*$ is
closed; then $ 0_S \cap \mathrm{cl}(G \cdot p^*) = \emptyset$, since $0_S$ is
$G$-invariant.  By the first part, $p$ is semi-stable, so we can find an
invariant section $\sigma\in H^0(X, \sheaf{L}^d)$ with $d > 0$ such that
$\sigma(p^*)$ is nonzero.  By scaling $\sigma$, we may assume $\sigma(p^*)=1$.
Let
\begin{equation*}
W = \Spec T/(\sigma-1) \subset X^*\setminus Z.
\end{equation*}
The open locus $\sigma\ne 0$ in $X$ is
\begin{equation*}
X_\sigma = \Spec T^{(d)}/(\sigma-1)
\end{equation*}
(where $T^{(d)} = \bigoplus_n T_{nd}$).  Then $T/(\sigma-1)$ is finite as a
module over its subalgebra $T^{(d)}/(\sigma-1)$, hence $W\to X_\sigma$ is
finite and surjective.

Thus we have a diagram
\begin{equation*}
\begin{diagram}
G & \rTo^{\phi_{p^*}} & W \\
  & \rdTo_{\phi_p} & \dTo_\pi \\
  && X_\sigma
\end{diagram}
\end{equation*}
in which $\pi$ is finite, and hence proper. Then $\phi_p$ is proper if and only
if $\phi_{p^*}$ is proper.  We have assumed $\phi_{p^*}$ is proper, so $\phi_p$
is proper.  By Lemma \ref{lemma:proper-phi}, $G\cdot p$ is closed in
$X_\sigma$, and $G_p$ is finite. By Proposition \ref{prop:variations}(3), $p$
is stable.

Conversely, if $p$ is stable, then it is certainly semi-stable, so again we
have the diagram above.  By Proposition \ref{prop:variations}(2), $G\cdot p$
closed in $X_\sigma$ and $G_p$ is finite. So $\phi_p$ is proper by Lemma
\ref{lemma:proper-phi}, and thus $\phi_{p^*}$ is proper.
\end{proof}

\section{The numerical criterion --- proof}\label{sec:numerical}

We now use the geometric criteria from the previous section to establish the
numerical criterion in Theorem \ref{theo:numericalcriterion}.

\begin{definition}\label{def:linear}
By a \emph{linear} action of an affine group $G=\Spec B$ on $\VV(V)$, we mean
an action
\begin{equation*}
G\times \VV(V) \to \VV(V)
\end{equation*}
such that the image of $V\subset \Sym V$, under the corresponding ring
homomorphism $\Sym V \to B\tensor \Sym V$, is contained in $B\tensor V$.
\end{definition}

The action on $\VV(V) = \VV(H^0(X,\sheaf{L}))$ induced by the linearization of
$\sheaf{L}$ is (clearly) linear.

\begin{lemma}\label{lemma:grading}
Let $A$ be a ring and $V$ an $A$-module.
\begin{enumerate}
\item There is a canonical one to one correspondence between $\Gm$-actions on
$\Spec A$ and gradings $A=\bigoplus_{d\in \ZZ} A_d$.  \item Given an action as
in (1), there is a canonical one to one correspondence between linear actions
on $\VV(V)$, such that the projection $\VV(V)\to \Spec A$ is equivariant, and
gradings $V=\bigoplus_{d\in \ZZ} V_d$ making $V$ a graded $A$-module.
\end{enumerate}
\end{lemma}

\begin{proof}
This is well known; we briefly sketch the argument.  A $\Gm$-action on $\Spec
A$ corresponds to a ring homomorphism
\begin{equation*}
A \to A[t,t^{-1}]
\end{equation*}
and we let $A_d\subset A$ consist of all elements $a\in A$ which are mapped to
$at^d$.  One shows $A=\bigoplus A_d$ by using the identity and associativity
group axioms.  A linear $\Gm$-action on $\VV(V)$, over a given action on $\Spec
A$, corresponds to an $A$-module homomorphism
\begin{equation*}
V \to V[t,t^{-1}]
\end{equation*}
(where $V[t,t^{-1}]$ is an $A$-module via the given action $A\to A[t,t^{-1}]$)
and we let $V_d\subset V$ consist of all elements $v\in V$ which are mapped to
$vt^d$.  One shows $V=\bigoplus V_d$ by using the identity and associativity
group axioms, and $A_dV_e\subset V_{d+e}$ is immediate.
\end{proof}

In the absolute case of a projective scheme over $k$, any action of a one
parameter subgroup on $\PP^n_k$ can be diagonalized. To handle the relative
situation $\PP(V)\to S$, we next introduce a construction that will replace the
coordinates of a point with respect to a diagonalizing basis.

An $R$-valued point $v\colon \Spec R \to \VV(V)$ is equivalent to a pair
\begin{align*}
A &\to R  \quad \text{($k$-algebra homomorphism)}\\
V &\to R  \quad \text{($A$-module homomorphism)}.
\end{align*}
In fact, the first homomorphism $A\to R$ defines the base point
\begin{equation*}
f(v)\colon \Spec R \to \Spec A,
\end{equation*}
and the second homomorphism $V\to R$ (in which $R$ is considered as an
$A$-algebra through $f(v)$) determines and is determined by an $A$-algebra
homomorphism $\Sym V\to R$, which defines $v$.  When $\Gm$ acts on $\Spec A$
and $\VV(V)$ compatibly, so that we have gradings $A=\bigoplus A_d$ and
$V=\bigoplus V_d$, we shall write
\begin{equation*}
[v]_d \colon V_d \to R
\end{equation*}
for the composition $V_d \subset V \to R$ (a $k$-module homomorphism).
Together, $f(v)$ and all $[v]_d$ determine the point $v$ uniquely. Similarly,
we shall write
\begin{equation*}
[f(v)]_d \colon A_d \to R
\end{equation*}
for the composition $A_d\subset A \to R$.

The effect of the $\Gm$-action on $[v]_d$ can be described as follows:
an $R$-valued point $g$ of $\Gm = \Spec k[t,t^{-1}]$ is a $k$-algebra homomorphism
\begin{equation*}
k[t,t^{-1}] \to R,
\end{equation*}
i.e.\ an invertible element $T\in R$ (the image of $t$). Its action on $v\in
\VV(V)(R)$ is the $R$-valued point
\begin{equation*}
\Spec R \xrightarrow{(g,v)} \Gm \times_k \VV(V) \to \VV(V)
\end{equation*}
or dually:
\begin{equation*}
\Sym V \to (\Sym V)[t,t^{-1}] \to R.
\end{equation*}
Its restriction
\begin{equation*}
V_d \to V_d t^d \to R
\end{equation*}
is $T^d[v]_d$, i.e. the homomorphism $[v]_d\colon V_d \to R$, followed by
multiplication by $T^d$ in $R$. Thus we have
\begin{equation*}
[g\cdot v]_d = T^d[v]_d
\end{equation*}
and similarly
\begin{equation*}
[g\cdot f(v)]_d = T^d[f(v)]_d.
\end{equation*}
Together, these two expressions determine $g\cdot v$ uniquely. For the
canonical $k(\!(t)\!)$-valued point in $\Spec k[t,t^{-1}]$ given by
\begin{equation*}
k[t,t^{-1}]\subset k(\!(t)\!),
\end{equation*}
the invertible element $T$ above is $t\in k(\!(t)\!)$.

We next show that the numerical invariant $\mu^\sheaf{L}(p,\lambda)$ in the
case of $\sheaf{L}=\OO(1)$ on $\PP(V)$ (for a linear action as in Definition
\ref{def:linear}), can be expressed in terms of the (non) vanishing of
$[p^*]_d$ for a lift $p^*\in\VV(V)$ of $p\in\PP(V)$. This is our replacement
for \cite[Prop.~2.3]{GIT}, expressing the numerical invariant on $\PP^n_k$
using coordinates with respect to a diagonalizing basis.

\begin{lemma}\label{lem:grading}
For a closed point $p\in \PP(V)$, and any lift $p^*\in \VV(V)$, we have
\begin{equation*}
\mu^{\OO(1)}(p, \lambda) = -\min \{d  \mid \text{$[p^*]_d\colon V_d\to k$ is nonzero}\}
\end{equation*}
for the grading $V=\bigoplus V_d$ induced by $\lambda$. We allow the
possibility $\mu = \infty$.
\end{lemma}

\begin{proof}
We divide the proof in two parts: first we treat the case where $\mu$ is
infinite.

By definition, $\mu^{\OO(1)}(p,\lambda)=\infty$ if and only if the limit
$\lim_{t\to 0} \lambda(t)\cdot p$ does not exist. As $\PP(V)\to S$ is
proper, this happens if and only if the limit $\lim_{t\to 0} \lambda(t)\cdot
f(p)$ in $S$ does not exist, by the valuative criterion for properness. The
limit in $S$ exists if and only if $[f(p)]_d=0$ for all $d<0$. We want to show
that this is equivalent to the statement that for any $d_0$, there exists
$d<d_0$ such that $[p^*]_d\ne0$.

Suppose the limit in $S$ does not exist. Then there is a homogeneous element
$a\in A_{-e}$ for some $e>0$, such that $[f(p)]_{-e}\colon A_{-e}\to k$ is
nonzero on $a$. Let $p^*$ be a lift of $p$; it is a point in $\VV(V)\setminus 0_S$,
and we have $[p^*]_d\ne 0$ for some $d$, since otherwise $p^*$ would be in the
zero section $0_S$. Pick any $v\in V_d$ such that $[p^*]_d(v) \ne 0$. Now, for
every $N>0$, the product $a^Nv \in V_{d-Ne}$ is an element on which
$[p^*]_{d-Ne}$ is nonzero.

Suppose the limit in $S$ exists, so that $[f(p)]_d=0$ for all $d<0$. The
$A$-module $V$ is finitely generated, say by homogeneous elements
$g_i\in V_{d_i}$. Now an arbitrary element in $V_d$ can be written $\sum_i a_i
g_i$ where $a_i\in A_{d-d_i}$. Since the homomorphism $V\to k$ defining $p^*$
is $A$-linear, it sends $a_ig_i$ to the product of $[f(p)]_{d-d_i}(a_i)$ and
$[p^*]_{d_i}(g_i)$. But $[f(p)]_{d-d_i}=0$ for all $d<d_i$. Thus $[p^*]_d=0$
for $d$ smaller than the degree of all the finitely many generators $g_i$.

Next we treat the case where the invariant $\mu$ is finite, i.e.~the limit
$p_0 = \lim_{t\to 0} \lambda(t)\cdot p$ exists.

The universal rank $1$ quotient
\begin{equation*}
V\tensor_A \OO_{\PP(V)} \to \OO_{\PP(V)}(1)
\end{equation*}
induces a closed immersion
\begin{equation*}
L=\VV(\OO_{\PP(V)}(1)) \subset \VV(V\tensor_A\OO_{\PP(V)}) = \VV(V)\times_A \PP(V)
\end{equation*}
and $L$ is the line bundle corresponding (contravariantly) to the invertible
sheaf $\OO_{\PP(V)}(1)$. This immersion is equivariant with respect to the
given $G$-actions on $L$ (via the linearization) and on $\VV(V)$ and $\PP(V)$.
Over the limit point $p_0$, we have fibres
\begin{equation*}
L(p_0) \subset \VV(V)(f(p_0)).
\end{equation*}
Let $r$ denote the minimal $d$ such that $[p^*]_d$ is nonzero. We shall find a
lift $p_0^*\in \VV(V)$ of $p_0$ (necessarily in $L(p_0)$, by definition of
``lift'') and show that $g\in\Gm$ acts on $p_0^*$ by scalar multiplication with
$g^r$.

For an arbitrary lift $p^*$ of $p$, we have $[p^*]_r \ne 0$ and $[p^*]_d=0$ for
$d<r$. The action of $\Gm$ on $\VV(V)$ corresponds to a grading
$V = \bigoplus V_d$. Now modify this action by shifting the degrees by $r$:
let $V= \bigoplus_d V'_d$, where $V'_d = V_{d+r}$ has degree $d$. This does not
change the action on $\PP(V)$, i.e.\ we have only changed its linearization.
With respect to this new grading, the ``homogeneous components'' of $p^*$ are
$[p^*]'_d=[p^*]_{d+r}$, so that $[p^*]'_d=0$ for all $d < 0$. The ring
homomorphism
\begin{equation*}
\Sym V \to k[t,t^{-1}]
\end{equation*}
corresponding to the (modified) action $\Gm\to\VV(V)$ on $p^*$ has image in
$k[t]$, since both $A=\Sym^0 V$ (by assumption on $f(p)$) and $V=\Sym^1 V$ have
image in $k[t]$. The specialization $p^*_0$ to $t=0$ is a lift of $p_0$.
Moreover, returning to the original grading $V=\bigoplus_d V_d$, we have
\begin{equation*}
[p^*_0]_d =
\begin{cases}
[p^*]_r & \text{for $d=r$}\\
0 & \text{otherwise.}
\end{cases}
\end{equation*}
Hence $g\in \Gm$ acts on $p^*_0$ by scalar multiplication with $g^r$.
\end{proof}

We shall abbreviate $\mu^{\OO(1)}(\lambda, p)$ to $\mu(\lambda,p)$ in the
following.

By the geometric criteria in Proposition \ref{prop:geometric-criteria}, the
Hilbert--Mumford criterion in Theorem \ref{theo:numericalcriterion} reduces to
the following:

\begin{theorem}\label{theo:main}
Let $p\in \PP(V)$ and let $p^*\in \VV(V)$ be an arbitrary lift.
\begin{enumerate}
\item $\phi_{p^*}\colon G \to \VV(V)$ is not proper if and only if there is a
field extension $k\subset K$ and a nontrivial $1$-PS $\lambda\colon \mathbb{G}_{m, K} \to G_K$ such that $\mu(\lambda, p) \le 0$.
\item The zero section $0_S\subset \VV(V)$ intersects the closure of the orbit
$G\cdot p^*$ non-trivially if and only if there is a field extension $k\subset K$ and a $1$-PS
$\lambda\colon \mathbb{G}_{m, K} \to G_K$, such that $\mu(\lambda, p) < 0$.
\end{enumerate}
Moreover, if $S$ is of finite type over $k$, then the above statements continue
to hold with just $K=k$.
\end{theorem}

Some notation in preparation for the proof: there are inclusions of $k$-algebras
\begin{equation*}
k \subset k[t] \subset k[\![t]\!] \subset k(\!(t)\!)
\end{equation*}
and
\begin{equation*}
k \subset k[t] \subset k[t,t^{-1}] \subset k(\!(t)\!).
\end{equation*}
($k[\![t]\!]$ is the algebra of formal power series, with the field
$k(\!(t)\!)$ of formal Laurent series as quotient field.) Thus any valued point
(in a separated scheme), with values in any of these algebras, can be
considered as a $k(\!(t)\!)$-valued point, and we may ask whether such a point
in fact takes values in one of the smaller algebras. If it takes value in (at
least) $k[\![t]\!]$, we may consider the limit as $t\to 0$, i.e.\ the
composition with $\Spec k \to \Spec k[\![t]\!]$ induced by $t\mapsto 0$.

In particular, a $1$-PS $\lambda\colon \Gm\to G$ is a $k[t,t^{-1}]$-valued
point, and Mumford writes $\langle \lambda \rangle$ for the same point
considered as $k(\!(t)\!)$-valued (we still write $\lambda$).

In the following proofs, we follow Mumford \cite[Theorem 2.1]{GIT} closely,
with our $[p^*]_d$ playing the role of the diagonalization of the $\Gm$-action.
The labelling (A) and (B) of equations below is for easy comparison with
Mumford.

\begin{proof}[Proof of ``only if'' in part (1)]
Suppose $\phi_{p^*}\colon G\to \VV(V)$ is not proper. By the valuative
criterion, there is a field extension $k\subset K$ and a commutative diagram 
\begin{equation}\label{diag:valuative} 
\begin{diagram}
\Spec \Kls & \rTo^\phi & G \\
\dTo && \dTo_{\cdot p^*} \\
\Spec \Kps & \rTo & \VV(V)
\end{diagram}
\end{equation}
where $\phi$ is not $\Kps$-valued (using that every complete DVR over $k$ is of
the form $\Kps$). Moreover, if $S$ is of finite type, then so is $\VV(V)$, and
then we may take $\Kps$ to be the completed local ring at a closed point of a
smooth curve over $k$, and thus assume $k=K$.

By the Iwahori theorem, there exists a $1$-PS $\lambda\colon \Gm \to G$ and two
$\Kps$-valued points $\psi_1, \psi_2 \in G(\Kps)$ such that
\begin{equation*}
\phi = \psi_1 \cdot \lambda \cdot \psi_2
\end{equation*}
(this is the product in the group $G(\Kls)$). Since $\phi$ is not $\Kps$-valued
(this is the only use of that fact), the $1$-PS $\lambda$ is nontrivial, i.e.\
not $k$-valued. Let $b_i = \psi_i(0)$ be the specializations to $t=0$ (we shall
only use $b_2$).

We shall show that $\mu(b_2^{-1}\cdot\lambda\cdot b_2, p) \le 0$. In the group
$G(\Kls)$, we have
\begin{equation*}
(\psi_1\cdot b_2)^{-1} \cdot \phi = (b_2^{-1}\cdot\lambda\cdot b_2)\cdot (b_2^{-1}\cdot \psi_2)
\end{equation*}
and when this acts on $p^*$, viewed as an element of $\VV(V)(\Kls)$, we get
\begin{equation}\label{eq:action}
(\psi_1\cdot b_2)^{-1} \cdot (\phi\cdot p^*) = (b_2^{-1}\cdot\lambda\cdot b_2)\cdot (b_2^{-1}\cdot \psi_2)\cdot p^*
\end{equation}
and here $\phi\cdot p^*$ is $\Kps$-valued, by diagram \eqref{diag:valuative}.
Thus everything on the left is $\Kps$-valued.

Now we use the grading $V=\bigoplus V_d$ corresponding to the $1$-PS
$b_2^{-1}\cdot \lambda \cdot b_2$. Then \eqref{eq:action} gives
\begin{align*}
[(\psi_1\cdot b_2)^{-1} \cdot (\phi\cdot p^*)]_d
&= [(b_2^{-1}\cdot\lambda\cdot b_2)\cdot (b_2^{-1}\cdot \psi_2)\cdot p^*]_d\\
&= t^d[(b_2^{-1}\cdot \psi_2)\cdot p^*]_d
\end{align*}
where multiplication by $t^d$ happens in $\Hom(V_d, \Kls)$. For this to be
$\Kps$-valued, we must have
\begin{equation*}
[(b_2^{-1}\cdot \psi_2)\cdot p^*]_d
\in t^{-d} \Hom(V_d, \Kps).\eqno{(A)}
\end{equation*}
On the other hand, $b_2^{-1}\cdot\psi_2$ specializes to the identity
$e\in G(K)$ at $t=0$, so $b_2^{-1}\cdot\psi_2\cdot p^*$ specializes to $p^*$ at
$t=0$. Thus
\begin{equation*}
[(b_2^{-1}\cdot \psi_2)\cdot p^*]_d = [p^*]_d + t Z_d \in \Hom(V_d, \Kps)\eqno{(B)}
\end{equation*}
for some $Z_d\in \Hom(V_d,\Kps)$.

Equations (A) and (B) imply that either $[p^*]_d = 0$ or $d\ge 0$. This proves
that $\mu(b_2^{-1}\cdot\lambda\cdot b_2, p) \le 0$.
\end{proof}

\begin{proof}[Proof of ``only if'' in part (2)]
We first note that the zero section $0_S\subset \VV(V)$ is $G$-invariant, so if
it intersects the closure of $G\cdot p^*$, then the orbit is not closed. This
prevents $\phi_{p^*}\colon G\to \VV(V)$ from being proper, and this particular
way of failing properness is detected by the valuative criterion by a diagram
\eqref{diag:valuative} in which $\phi\cdot p^*$ specializes to a point in the
zero section $0_S$ at $t=0$. Again, in the case where $S$ is of finite type
over $k$, there is a diagram \eqref{diag:valuative} with $k=K$.

Consequently (with notation as in the proof of part (1)),
$(\psi_1\cdot b_2)^{-1}\cdot(\phi\cdot p^*)$ is not only $\Kps$-valued, but
specializes to a point in $0_S$ at $t=0$ (since $\phi\cdot p^*$ does so, and
$0_S$ is invariant, so $(\psi_1\cdot b_2)^{-1}(0)$ must send
$(\phi\cdot p^*)(0)$ to another point in $0_S$). This means that
\begin{equation*}
[(\psi_1\cdot b_2)^{-1}\cdot(\phi\cdot p^*)]_d \in t \Hom(V_d,\Kps)
\end{equation*}
and then, instead of just (A) above, we get
\begin{equation*}
[(b_2^{-1}\cdot \psi_2)\cdot p^*]_d
\in t^{-d+1} \Hom(V_d, \Kps).\eqno{(\widetilde{A})}
\end{equation*}
Combined with (B) as before, we conclude that either $[p^*]_d=0$ or $d>0$, so
$\mu(b_2^{-1}\cdot \lambda \cdot b_2, p) < 0$.
\end{proof}

\begin{proof}[Proof of ``if'' in (1) and (2).]
Use the grading $V=\bigoplus V_d$ induced by the $1$-PS $\lambda$. The
$k[t,t^{-1}]$-valued point $\lambda\cdot p^*$ is uniquely determined by
$\lambda\cdot f(p)$ (a $k[t,t^{-1}]$-valued point in $S$) together with
\begin{equation*}
[\lambda\cdot p^*]_d = t^d[p^*]_d
\end{equation*}
for all $d$. When $\mu(\lambda, p) \le 0$, we have $[p^*]_d=0$ for all $d<0$,
This, together with $\lambda\cdot f(p)$ being $k[t]$-valued, implies that
$\lambda\cdot p^*$ is $k[t]$-valued, i.e.\ it has a specialization at $t=0$.

On the other hand, a nontrivial $1$-PS $\lambda\colon \Gm \to G$ is never
$k[t]$-valued, since the corresponding $k$-algebra-homomorphism must be
compatible with the group inverse $\iota\colon G\to G$:
\begin{equation*}
\begin{diagram}
\Gamma(G,\OO_G) & \rTo &k[t,t^{-1}] \\
\dTo^{\iota^*} & & \dTo_{t\mapsto t^{-1}}\\
\Gamma(G,\OO_G) & \rTo &k[t,t^{-1}]
\end{diagram}
\end{equation*}
Thus, the image $\lambda^*(\Gamma(G,\OO_G))$ is invariant under $t\mapsto
t^{-1}$, but then the image is either $k$ or $k[t,t^{-1}]$. It follows that
$\lambda$ gives rise to a diagram \eqref{diag:valuative}, preventing
properness. This proves ``if'' in claim (1).

When $\mu(\lambda, p) < 0$, we have $[p^*]_d=0$ for $d\le 0$. The
specialization $p^*_0$ of $\lambda\cdot p^*$ to $t=0$ thus satisfies $[p^*_0]_d
= 0$ for all $d$, so $p^*_0$ is a point in the zero section $0_S$. It is
clearly in the closure of $G\cdot p^*$. This proves ``if'' in claim (2).
\end{proof}

\section{Examples}\label{sec:examples}

\subsection{An elementary example} \label{ex:elementary}
We start with a first example where $X= \PP^1 \times \AA^2, S=\AA^2$  and the
map $f: X \to S$ is given by projection onto the second factor. We consider the
action of the  group $G=\Gm=\Spec k[t,t^{-1}]$ on $X$ by $((u:v),x,y)
\mapsto ((tu:v/t),tx,y/t)$ and on $S$ by  $(x,y) \mapsto (tx,y/t)$. Clearly,
$f$ is then $G$-equivariant. The group action of $G$ on $X$ can be linearized
by replacing $\PP^1$ by $\AA^2$ with affine coordinates $(u,v)$ and the action
$(u,v) \mapsto (tu,v/t)$. We shall work with the $G$-linearized line bundle
$\sheaf{L}$ given by the pullback of $\sheaf{O}_{\PP^1}(1)$ and we shall first
analyse the  (semi-)stable and unstable points of $X$  by using Theorem
\ref{theo:numericalcriterion}. According to this theorem points $p \in X$ which
are not stable are detected by non-trivial $1$-parameter subgroups $\lambda$ of
$G$. A necessary condition for a point $p$ not to be stable is that the limit
$\lim_{t \to 0} \lambda(t)f(p)$ exists. This shows immediately that points $p$
with $x,y \neq 0$ are always stable. Now assume $y=0$. For $\lim_{t \to 0}
\lambda(t)f(p)$ to exist we must have $\lambda(t)=t^d$ with $d>0$. If
$v\neq 0$, then $\mu(\lambda,p) = d >0$ and thus there cannot be a
destabilizing $1$-parameter subgroup. On the other hand if $v=0$, then
$\mu(\lambda,p) =  -d <0$ and hence these points are unstable. The same
argument shows that the points $x=u=0$ are unstable, whereas the points
$x=0, u \neq 0$ are stable. In this example there are no strictly semi-stable
points.

The good quotient $S/G$ is easy to understand. Let $A=k[x,y]$. Then $A^G=k[T]$
where $T=xy$. The fibres of the map $S \to \Spec k[x,y]^G=\Spec k[T]=\AA^1$
away from the origin are the closed orbits $xy=T\neq 0$, whereas the coordinate
axes are mapped to the origin. Next we turn to $X=\Proj A[u,v]$. We have
$A[u,v]^G = k[xy, xv, yu, uv]$. Renaming these generators by $T=xy$, $X=xv$,
$Y=yu$, $Z=uv$ we find that $A[u,v]^G= A^G[X,Y,Z]/(XY-TZ)$. Thus $X^\stab/G \to
\Spec A^G$ is a conic bundle with two lines intersecting transversally over the
origin. The singular point in the central fibre is the image of points $x=y=0$
and $uv \neq 0$. The orbits of stable points with $x=0, y \neq 0$ and $y=0,
x\neq 0$ are mapped to the smooth points of the two components of the central
fibre . We remark that $X^\stab/G$ lives in $\AA^1 \times \PP(1,1,2)$ and has
an ordinary $A_1$-singularity at $X=Y=T=0$, i.e. where the two lines in the
central fibre meet. The singularity arises since the points in the orbit
$x=y=0$, $uv\ne 0$ have nontrivial stabilizer group $\mu_2$.

\subsection{A degeneration of Hilbert schemes on conics}\label{ex:degHilbert}

The next example illustrates how our results can be used to study Hilbert
schemes for a semi-stable degeneration of curves. In a subsequent paper we
shall use this approach to study degenerations of Hilbert schemes more
generally, notably in the setting of $K3$ surfaces.

We put $X = V(t z^2 - xy) \subset \AA^1 \times \PP^2$. The semi-stable fibration
\begin{equation*}
f \colon X \to \AA^1 = \Spec k[t]
\end{equation*}
degenerates a smooth conic into a union of two lines $Y = V(t,x)$ and $Y' =
V(t,y)$ intersecting transversally in a point $\mathfrak{p}$. The Hilbert
scheme $\Hilb^n_{X/\AA^1}$ and its central fibre have quite bad singularities
due to the presence of subschemes with $\mathfrak{p}$ contained in the support.
We shall construct a better behaved compactification of
$\Hilb^n_{X^{\mathrm{sm}}/\AA^1} \to \AA^1$, where
$X^{\mathrm{sm}}= X - \{\mathfrak{p}\}$. Indeed, the stack of Li--Wu
\cite{LiWu} yields such a compactification; we shall give a GIT version of
their construction.

Roughly, the idea is to replace $ \mathfrak{p} $ by a chain of smooth rational
curves linking $Y$ and $Y'$, and add certain subschemes of length $n$ supported
in the \emph{smooth} locus of the resulting semi-stable curve.

An important tool is Li's so-called \emph{expanded degenerations}
\begin{equation*}
f[n] \colon X[n] \to \AA^{n+1}
\end{equation*}
of the map $f$. We will start by explaining a few key points of this theory,
for more details see \cite{Li} and \cite{ACFW} (where it is also explained that
the final result is independent of the various choices which one makes in the
construction).

Let $s$ and $s'$ be defining sections of the effective divisors $Y$ and $Y'$,
and let $\AA^2 \to \AA^1$ be the map $(t_1, t_2) \mapsto t_1t_2$. Then
$X \times_{\AA^1} \AA^2$ has an $A_1$-singularity, locally of the form
$s s' - t_1 t_2 = 0$. The blow up
$\pi_1 \colon X[1] \to X \times_{\AA^1} \AA^2$ in $Y \times V(t_2)$ resolves
this singularity by a small resolution, and the induced map
$X[1] \to \AA^2$ is $f[1]$.

The scheme $X[1]$ is a closed subscheme of the $\PP^1$-bundle
\begin{equation*}
\PP[1] := \PP(p_1^*\OO_{X} (-Y) \oplus \OO_{X \times_{\AA^1} \AA^2})
\end{equation*}
over $X \times_{\AA^1} \AA^2$, cut out by equations $s v_1 - t_2 u_1 = 0$ and
$t_1 v_1 - s' u_1 = 0$, where $u_1$ and $v_1$ denote homogeneous coordinates in
the fibres.

We let the torus $G[1] := \Spec k[\sigma_1, \sigma_1^{-1}]$ act on $\AA^2$ by
$(t_1,t_2) \mapsto (\sigma_1 t_1, \sigma_1^{-1} t_2)$. Then $Y \times V(t_2)$
is invariant for the induced action on $X \times_{\AA^1} \AA^2$, hence the
action lifts uniquely to $X[1]$. Letting $G[1]$ act on $\PP[1]$ by $(u_1:v_1)
\mapsto (\sigma_1 u_1:v_1)$; also $X[1] \subset \PP[1]$ is equivariant.

Composing $f[1]$ with the projection from $\AA^2$ to its $t_2$-factor gives a
map $X[1] \to \AA^1$, the central fibre is again a union of smooth irreducible
components $Y_1$ and $Y_1'$ intersecting transversally. This fact allows one to
iterate the above procedure; inductively, one constructs $X[n]$ as $X[n-1][1]$.
Again, $X[n]$ is a closed subscheme of a $\PP^1$-bundle $\PP[n]$ over
$X[n-1] \times_{ \AA^{n} } \AA^{n+1}$, inducing a natural fibration
$f[n] \colon X[n] \to \AA^{n+1}$.

The torus $G[n] := \Spec k[\sigma_1^{\pm 1}, \ldots, \sigma_n^{\pm 1}]$ acts on
$f[n]$, compatibly with the $G[n-1]$-action on $f[n-1]$; on $\AA^{n+1}$ by
\begin{equation*}
(t_1, \ldots, t_i, \ldots, t_{n+1}) \mapsto (\sigma_1 t_1, \ldots, \sigma_i \sigma_{i-1}^{-1} t_i, \ldots, \sigma_n^{-1} t_{n+1})
\end{equation*}
and in the $\PP^1$-fibres of $\PP[n]$ by $(u_n:v_n) \mapsto (\sigma_n u_n:v_n)$.

We introduced the $\PP^1$-bundles above in order to construct a useful
$G[n]$-linearized sheaf $\sheaf{L}$ on $X[n]$. Fix an ample sheaf $\OO_X(1)$ on
$X$ and let $\OO_{\pi_i}(1)$ denote the restriction to $X[i]$ of the hyperplane
bundle on $\PP[i]$. For suitable integers $0 \ll a_{n-1} \ll \ldots \ll a_0$,
the invertible sheaf
\begin{equation*}
\sheaf{L}_1 := \OO_{\pi_n}(1) \otimes g_{n-1}^* \OO_{\pi_{n-1}}(a_{n-1}) \otimes \ldots \otimes g_0^* \OO_{X}(a_0)
\end{equation*}
is very ample for $f[n]$ \cite[8.1.22]{Liu}. Here $g_i$ denotes the obvious map
from $X[n]$ to $X[i]$ (with $X[0] := X$).

For each $i$, we define a $G[i]$-linearization on $\OO_{\pi_i}(n+1)$ by
\begin{equation*}
u_i^{\ell} v_i^{n+1-\ell} \mapsto \sigma_i^{i-(n+1)+\ell} u_i^{\ell} v_i^{n+1-\ell}
\end{equation*}
for $0 \leq \ell \leq n+1$, where the monomials are viewed as sections of
$\OO_{\pi_i}(n+1)$. This induces canonically a $G[n]$-linearization on
$\sheaf{L}_1^{n+1}$. Linearizing $\sheaf{L}_1^{n+1}$ rather than $\sheaf{L}_1$
is quite crucial; it will yield a better semi-stable locus in the Hilbert
scheme later on. We need also to take a suitable $N$-th power of this sheaf,
thus in the end we put $\sheaf{L} := \sheaf{L}_1^{(n+1)N}$. (The integers $a_0,
\ldots, a_{n-1}$ and $N$ play only a formal role.)

Consider finally $\mathcal{H}[n] := \Hilb^n _{X[n]/\AA^{n+1}}$, and its universal family
\begin{equation*}
\mathcal{Z} \subset \mathcal{H}[n] \times_{\AA^{n+1}} X[n].
\end{equation*}
For $N$ sufficiently large, the line bundle
$\sheaf{M} := \det p_{1*}\res{(p_2^* \sheaf{L})}{\mathcal{Z}}$ on $\mathcal{H}[n]$
is relatively very ample \cite[Prop.~2.2.5]{HuLe} and inherits a linearization
from $\sheaf{L}$. Proposition \ref{prop:proj} now asserts that the map between
the good quotients
\begin{equation*}
\mathcal{H}^n_{X/\AA^1} := \mathcal{H}[n]^{ss}(\sheaf{M})/G[n] \to \AA^{n+1}/G[n] = \AA^1
\end{equation*}
is \emph{projective}. This is our compactification of $
\Hilb^n_{X^{\mathrm{sm}}/\AA^1} \to \AA^1 $.

Even though the global structure of $\sheaf{M}$ might be complicated, we can
control the fibres of $\sheaf{M}$ sufficiently well in order to apply Theorem
\ref{theo:main}. As we shall see below, this will yield a neat description of
the (semi-)stable locus. Before we formulate this, we need to discuss the
fibres of $f[n]$.

The fibre $X[n]_0$ over the origin in $\AA^{n+1}$ forms a transversal chain
\begin{equation*}
Y \cong \Delta_0 \cup \ldots \cup \Delta_i \cup \ldots \cup \Delta_{n+1} \cong Y',
\end{equation*}
where $\Delta_i \cong \PP^1$ for $1 \leq i \leq n$. The torus $G[n]$ acts
trivially on $\Delta_0$ and $\Delta_{n+1}$, and otherwise scales a coordinate
of $\Delta_i$ with $\sigma_i$, keeping the intersection points fixed.
Generally, if $q$ belongs to the stratum where $t_i$ is zero precisely for the
indices $i_1 < \ldots < i_r$, then $X[n]_q$ has the form
\begin{equation*}
Y \cong \Delta_{I_0} \cup \ldots \cup \Delta_{I_k} \cup \ldots \cup \Delta_{I_r} \cong Y'
\end{equation*}
where $I_k = \{i_k, \ldots, i_{k+1}-1\}$ and $\Delta_{I_k}$ is a smoothing of
the ``sub-chain'' $\Delta_{i_k} \cup \ldots \cup \Delta_{i_{k+1}-1}$ as the
parameters indexed by $I_k$ become nonzero. We write $\Delta_{I_k}^{\circ}$ for
the component $\Delta_{I_k}$ with the double locus removed.

Analysing the stability criterion we arrive at the following description of the
semi-stable locus:

$\bullet$ A subscheme $Z \subset X[n]_q $ of length $n$ is semi-stable if and only if
\begin{equation*}
\mathit{length}(Z \cap \Delta_{I_k}^{\circ}) = \mathit{cardinality}(I_k \cap \{1, \ldots, n\} )
\end{equation*}
holds for every irreducible component $\Delta_{I_k}$ of $X[n]_q$. Moreover, any
semi-stable point is also stable (for details on the case $n=2$ see below).

Note, in particular, that a semi-stable $Z$ does not meet the singular locus of
$X[n]_q$. Thus the composition $\mathcal{H}[n]^\sstab \to \AA^{n+1} \to \AA^1$
consists of a smooth map followed by the multiplication
$t \mapsto t_1 \cdot \ldots \cdot t_{n+1}$. One easily deduces that
$\mathcal{H}^n_{X/\AA^1} \to \AA^1$ is, up to finite quotient singularities, a
semi-stable degeneration. Alternatively; the stack quotient
$[\mathcal{H}[n]^\sstab/G[n]] \to \AA^1 $ is a proper, semi-stable
degeneration. Note also that since
$\mathcal{H}[n]^\sstab = \mathcal{H}[n]^\stab$, the quotient
$\mathcal{H}^n_{X/\AA^{1}}$ is an orbit space.

We note that the characterization of GIT stability in $\mathcal{H}[n]$ we have
deduced using the numerical criterion, essentially coincides with the
definition of admissibility in \cite{LiWu}. Also, properness of the quotient
map is in our case an immediate consequence of Proposition \ref{prop:proj}.

We now discuss the case $n=2$ in some detail, to illustrate the results
mentioned above. We start by making a key observation. Let $Z$ be a length $2$
subscheme in a fibre of $f[2]$ and let $\lambda$ be a $1$-PS for which the
limit $Z_0$ of $Z$ exists. If $Z_0 = \{p_0,p_0'\} $ is reduced, then
\begin{equation*}
\sheaf{M}(Z_0) \cong \sheaf{L}(p_0) \otimes \sheaf{L}(p_0'),
\end{equation*}
thus the $\Gm$-weight on $\sheaf{M}(Z_0)$ can be computed from the weights on
$\sheaf{L}$ at $p_0$ and $p_0'$. If $Z_0$ is supported in a single point $p_0$,
it is not true that $\sheaf{M}(Z_0) \cong \sheaf{L}(p_0)^{\otimes 2}$. However,
one can show that the error term one gets when replacing $\mu(\lambda,Z)$ by
twice the $\Gm$-weight on $\sheaf{L}(p_0)$ becomes negligible when $N \gg 0$.
Thus, in our computations, it suffices to consider formal sums of points in the
fibres of $f[2]$.

For $Z \subset X[2]_0$ this can be detailed as follows. Let $\lambda$ be a
$1$-PS of $G[2]$ corresponding to $(s_1,s_2) \in \ZZ^2 $. Fixing a coordinate
$\sigma$ for $\lambda$, the action in the $\PP^1$-fibres is
$(\sigma^{s_i}u_i:v_i)$. Let $p \in \mathrm{Supp}(Z)$ and put
$p_0 := \lim_{\sigma \to 0} \lambda(\sigma) \cdot p$. If
$p \in \Delta_j^{\circ}$ where $j \in \{0,3\}$, we have $p = p_0$. If however
$j \in \{ 1,2 \}$, both $u_j$ and $v_j$ are nonzero at $p$, and easy
computations show that $u_j = 0$ at $p_0$ if $s_j > 0$ and that $v_j = 0$ at
$p_0$ if $s_j < 0$.

Now we will compute the $\Gm$-weight $\mu(\lambda, p)$ on $\sheaf{L}(p_0)$.
Assume first that $p \in \Delta_1^{\circ}$. If $s_1 > 0$, then $\sheaf{L}(p_0)$
is generated by $(v_1^{3a_1} v_2^{3})^N$. If instead $s_1 < 0$, then
$\sheaf{L}(p_0)$ is generated by $(u_1^{3a_1}v_2^{3})^N$. In any case, we find
\begin{equation*}
\mu(\lambda, p) = N(-a_1 (1/2 s_1 - 3/2 \vert s_1 \vert ) - s_2).
\end{equation*}
Assuming instead that $p \in \Delta_2^{\circ}$, a similar computation yields
\begin{equation*}
\mu(\lambda, p) = N( a_1 s_1 + 1/2 s_2 - 3/2 \vert s_2 \vert).
\end{equation*}
Hence, for $Z = \{p_1, p_2 \}$ where $p_j \in \Delta_j^{\circ}$, we get
\begin{equation*}
\mu(\lambda,Z) = N(a_1(1/2 s_1 - 3/2 \vert s_1 \vert) - (1/2 s_2 + 3/2 \vert s_2 \vert)).
\end{equation*}
In particular, we find that $\mu(\lambda,Z) \leq 0$ with equality if and only
if $s_1 = s_2 = 0$, i.e., if $\lambda$ is trivial. This means that $Z$ is
stable.

Lastly, let us say a few words about the geometry of
$\mathcal{H}^2_{X/\AA^1} \to \AA^1$. Over the stratum in $\AA^3$ where
$t_1 = t_3 = 0$ and $t_2 \neq 0$, the fibres of $f[2]$ have a component
$\Delta_{ \{ 1,2 \} }$, where $\Delta_{ \{1,2\} }^{\circ} \cong \Gm$ and where
the diagonal torus in $G[2]$ acts by $w \mapsto \sigma \cdot w$ for a
coordinate $w$ of $\Gm$. Thus, the subschemes
$\{ p, -p \} \subset \Delta_{ \{1,2\} }^{\circ}$ have stabiliser $ \mu_2 $ and
one can check that the action is free outside this locus. This yields an
isolated $A_1$ singularity in $\mathcal{H}^2_{X/\AA^1}$. Our family degenerates
$\PP^2$ into a union of rational surfaces $H_{20}$, $H_{11}$ and $H_{02}$,
where the generic point on $H_{ij}$ corresponds to $i$ points on $Y^{\circ}$
and $j$ points on $(Y')^{\circ}$. One easily checks that the $H_{ij}$-s
intersect pairwise in an irreducible curve, with triple intersection a point,
i.e., the dual graph forms a $2$-simplex. For instance, to describe
$H_{2,0} \cap H_{0,2}$, note that both $\Delta_{\{0,1,2\}} \cup \Delta_3$ and
$\Delta_0 \cup \Delta_{\{1,2,3\}}$ are smoothings of
$\Delta_0 \cup \Delta_{\{1,2\}} \cup \Delta_3$ as $t_1$, resp.~$t_3$, becomes
nonzero. By our description of the (semi-)stable locus, this means that
$H_{20}$ and $H_{02}$ intersect precisely in the closure of the locus
consisting of stable subschemes of $\Delta_0 \cup \Delta_{\{1,2\}} \cup
\Delta_3$, modulo the action of the diagonal torus in $G[2]$.

\bibliographystyle{numsty}

\end{document}